\documentclass[twoside,12pt]{article}
\usepackage{color}

\usepackage{graphics}
\usepackage{graphicx}

\usepackage{amssymb,amsfonts}
\usepackage{amsmath}
\usepackage{amsthm}
\usepackage{latexsym}
\usepackage{fancyhdr,amssymb}
\usepackage{url}		
\usepackage{arydshln}

\graphicspath{{fig/},{figure/},{../fig/},{../fig_eps/}}

\def\binomh#1#2{ \scalebox{.3}[1.2]{\textbf{)}}{\genfrac{}{}{0pt}{}{#1}{#2}}\scalebox{.3}[1.2]{\textbf{(}} }
\def\binomhh#1#2{ \scalebox{.4}[1.7]{\textbf{)}}{\genfrac{}{}{0pt}{}{#1}{#2}}\scalebox{.4}[1.7]{\textbf{(}} }

\def\hpt4q5{${\cal HPT}_{\!4,5}$}
\def\hpt4q6{${\cal HPT}_{\!4,6}$}
\def\hpt4q{${\cal HPT}_{\!4,q}$}

\def\g{\bar{\bf g}}
\def\d{\bar{\bf d}}
\def\0{\bar{\bf 0}}

\makeatletter
\def\Ddots{\mathinner{\mkern1mu\raise\p@
\vbox{\kern7\p@\hbox{.}}\mkern2mu
\raise4\p@\hbox{.}\mkern2mu\raise7\p@\hbox{.}\mkern1mu}}
\makeatother

\newcommand {\N}{\mathbb{N}}

\newcommand {\R}{\mathbb{R}}

\newtheorem{lemma}{Lemma}[section]

\newtheorem{remark}{Remark}
\newtheorem{conj}{Conjecture}

\title{\bf Power sums in hyperbolic Pascal triangles
}

\author{L\'aszl\'o N\'emeth\footnote{University of Sopron,  Institute of Mathematics, Hungary. \textit{nemeth.laszlo@uni-sopron.hu}}, 
 L\'aszl\'o Szalay\footnote{University J.~Selye, Department of Mathematics and Informatics, Slovakia; and University of Sopron,  Institute of Mathematics, Hungary.  \textit{szalay.laszlo@uni-sopron.hu}}}
\date{}

\begin{document}

\maketitle

\begin{abstract}
 In this paper, we describe a method to determine the power sum of the elements of the rows in the hyperbolic Pascal triangles corresponding to $\{4,q\}$ with $q\ge5$. The method is based on the theory of linear recurrences, and the results are demonstrated by evaluating the $k^{th}$ power sum in the range $2\le k\le 11$.  \\[1mm]
 
 The final publication is available at Analele Stiintifice ale Universitatii Ovidius Constanta,  (www.anstuocmath.ro).  
\end{abstract}


\section{Introduction}\label{sec:introduction}

Binomial coefficients,  Pascal's triangle and its generalizations have been widely studied. One of the examined properties is the power sum 
\begin{equation*}
S_k(n)=\sum_{i=0}^n {n \choose i}^k, \quad n\geq 0,\quad k\in \{0,1,2,\ldots\}.
\end{equation*}
Apart from $S_0(n)=n+1$, $S_1(n)=2^n$, and $S_2(n)={2n \choose n}$, there is no general solution for $S_k(n)$, but recurrences are given for the cases $3\leq k\leq 10$. For example, Franel \cite{F1,F2} obtained the recurrence
$$
S_3(n+1)=\frac{7n^2+7n+2}{(n+1)^2}\,S_3(n)+\frac{8n^2}{(n+1)^2}\,S_3(n-1)
$$
for $k=3$.
The difficulty of the problem is also indicated by the result, that there is no closed form, only asymptotic formula for $S_k(n)$ in the cases $3\leq k\leq 9$.  For more details and further references of power sums of binomial coefficient see \cite{Dzh,Cusick,Yuanu}. 

In this article, we investigate the analogous question, and give a general method to determine the power sums in hyperbolic Pascal triangles, which are a recently discovered geometrical generalizations of Pascal's classical triangle. The method is illustrated by evaluating the $k^{th}$ power sum in the range $2\le k\le 11$, and we calculate the corresponding recurrences up to $k=11$. A hyperbolic Pascal triangle contains two types of elements (say $A$ and $B$) in the rows, therefore its structure is more complicated than the Pascal's triangle's, but at the same time the existence of elements having type $B$ facilitates the determination of power sums since they behave as separating items between elements having type $A$. A short summary on hyperbolic Pascal triangles can be found in the next subsection.

\subsection{Hyperbolic Pascal triangles}
 
In the hyperbolic plane there are infinite types of regular mosaics, they are denoted by Schl\"afli's symbol $\{p,q\}$, where the positive integers $p$ and $q$ has the property $(p-2)(q-2)>4$. Each regular mosaic induces a so-called hyperbolic Pascal triangle (see \cite{BNSz}), following and generalizing the connection between the classical Pascal's triangle and the Euclidean regular square mosaic $\{4,4\}$.
For more details see \cite{BNSz,NSz_alter,NSz_recu}, but here we also collect some necessary information. 

There are several approaches to generalize Pascal's arithmetic triangle (see, for instance \cite{BSz}).
The hyperbolic Pascal triangle based on the mosaic $\{p,q\}$ can be figured as a digraph, where the vertices and the edges are the vertices and the edges of a well defined part of the lattice $\{p,q\}$, respectively, further each vertex possesses the value which gives the number of different shortest paths from the base vertex. Figure~\ref{fig:Pascal_46_layer5} illustrates the hyperbolic Pascal  triangle ${\cal HPT}_{\!p,q}$, when $\{p,q\}=\{4,6\}$. 

\begin{figure}[h!]
	\centering
	\includegraphics[width=0.99\linewidth]{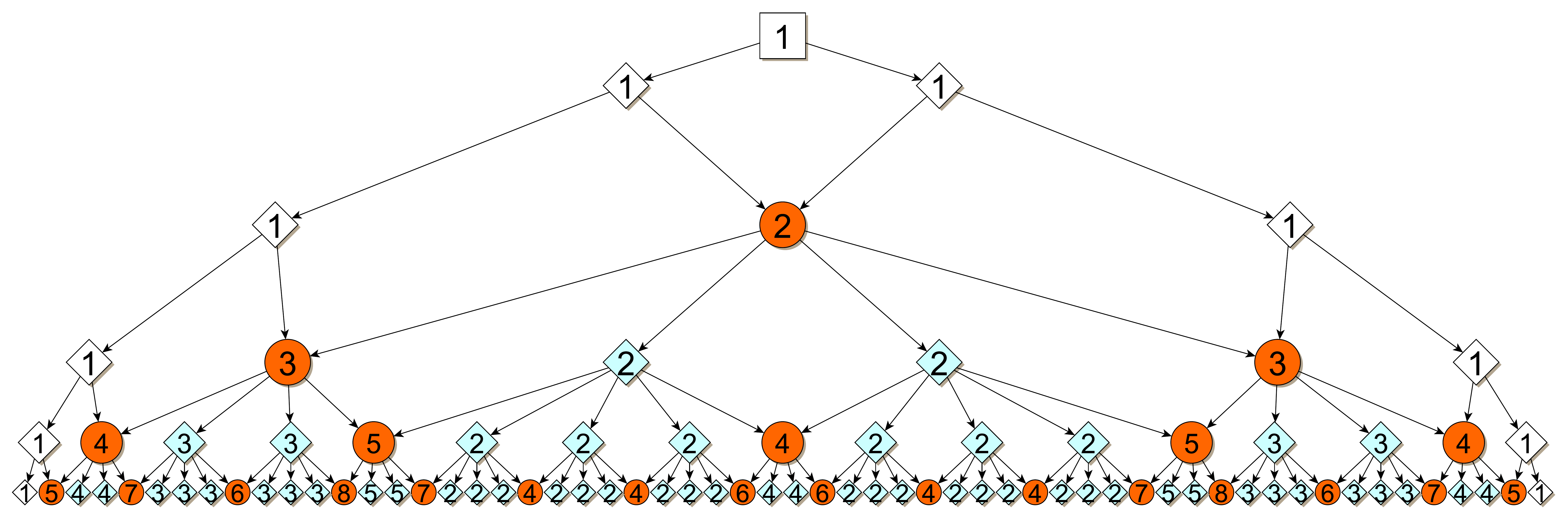}
	\caption{Hyperbolic Pascal triangle linked to $\{4,6\}$ up to row 5}
	\label{fig:Pascal_46_layer5}
\end{figure}

Generally, for $\{4,q\}$ the base vertex has two edges, the leftmost and the rightmost vertices have three, the others have $q$ edges. The square shaped cells surrounded by appropriate edges are corresponding to the squares in the regular mosaic.
Apart from the winger elements, certain vertices (called ``Type A'' for convenience) have two ascendants and $q-2$ descendants, the others (``Type B'') have one ascendant and $q-1$ descendants. In the figures, we denote the vertices type $A$ by red circles and the vertices type $B$ by cyan diamonds, further the wingers by white diamonds. The vertices which are $n$-edge-long far from the base vertex are in row $n$. 

The general method of drawing is the following. Going along the vertices of the $j^{th}$ row, according to type of the elements (winger, $A$, $B$), we draw appropriate number of edges downwards (2, $q-2$, $q-1$, respectively). Neighbour edges of two neighbour vertices of the $j^{th}$ row meet in the $(j+1)^{th}$ row, constructing a vertex with type $A$. The other descendants of row $j$ in row $j+1$ have type $B$.
In the sequel, $\binomh{n}{k}$ denotes the $k^\text{th}$ element in row $n$, which is either the sum of the values of its two ascendants or the value of its unique ascendant. We note, that the hyperbolic Pascal triangle has the property of vertical symmetry. 

In the remaining part of the paper we consider the winger nodes (having value 1) as elements with type $B$. Denote by $a_n$ and $b_n$ the number of vertices of type $A$ and $B$ in row $n$, respectively, further let 
\begin{equation*}\label{sn}
s_n=a_n+b_n, 
\end{equation*}
which gives the total number of the vertices of row $n\ge0$.
Then the ternary homogeneous recurrence relation
\begin{equation*}\label{recur1}
s_n=(q-1)s_{n-1}-(q-1)s_{n-2}+s_{n-3}\qquad (n\ge4),
\end{equation*}
holds with initial values 
$s_1=2$, $s_2=3$, $s_3=q$  (see \cite{BNSz}).

Moreover, let $\hat{a}_n$, $\hat{b}_n$ and $\hat{s}_n$ denote the sum of type $A$, type $B$ and all elements of the $n^{th}$ row, respectively. It was justified in \cite{BNSz} that the three sequences $\{\hat{a}_n\}$, $\{\hat{b}_n\}$ and $\{\hat{s}_n\}$ can be described by the same ternary homogenous recurrence relation 
\begin{equation*}\label{eq:hats}
x_n=qx_{n-1}-(q+1)x_{n-2}+2x_{n-3}\qquad (n\ge4),
\end{equation*}
their initial values are 
$$
\hat{a}_1=0,\;\hat{a}_2=2,\;\hat{a}_3=6;\quad \hat{b}_1=2,\;\hat{b}_2=2,\;\hat{b}_3=2q-6;\quad \hat{s}_1=2,\;\hat{s}_2=4,\;\hat{s}_3=2q.$$

\subsection{New results}

In this article, we describe a method to determine the $k^{th}$-power sum
\begin{equation*}\label{eq:SUMk}
(s^k)_{n}=\sum_{i=0}^{s_n-1}{\binomhh{n}{i}}^k.
\end{equation*}
We will illustrate the technique by the cases $2\le k\le 11$. The results are given by different linear recursions.  Note that $s_n=(s^0)_n$, further $\hat{s}_n=(s^1)_n$. Hence we interested in the problem if $k\ge2$.  

Here we also introduce the notation $(a^k)_{n}$ and $(b^k)_{n}$ for the sum of the $k^{th}$ power of elements of type $A$ and $B$ in row $n$, respectively. Clearly, $(s^k)_{n}=(a^k)_{n}+(b^k)_{n}$.

\section{The method}

Let us start with a further technical note. Take two consecutive elements ${\binomh{n}{\ell}}$ and  ${\binomh{n}{\ell+1}}$ of ${\cal HPT}_{\!4,q}$, and let denote $X$ and $Y$ the not necessarily distinct types of them, respectively. Let the $i^{th}$ and $j^{th}$ power-product of them is denoted by $x^iy^j={\binomh{n}{\ell}}^i\cdot{\binomh{n}{\ell+1}}^j$, and the sum of all such products (where the first term has type $X$, the second does $Y$) of row $n$ by $(x^iy^j)_n$. For example, in case of ${\cal{HPT}}_{\!4,6}$ we find $(a^2b^3)_{3}=3^2\cdot2^3+3^2\cdot1^3=81$ and $(bb^3)_{3}=2\cdot 2^3=16$ (see Figure~\ref{fig:Pascal_46_layer5}).
Clearly, by the vertical symmetry of \hpt4q, $(a^ib^j)_{n}=(b^ja^i)_{n}$ holds. An other easy but important observation is that $(b^{i-1}b)_n=(b^{i-2}b^2)_n=\cdots =(bb^{i-1})_n$, since two consecutive elements having type $B$ are equal. 

\subsection{Case $k=2$}

Before the general description, we restrict ourselves to the sum of squares to facilitate the justification of the method. Hence fix $k=2$, and take 
\begin{equation*}\label{eq:SUMkk}
(s^2)_{n}=\sum_{i=0}^{s_n-1}{\binomhh{n}{i}}^2= (a^2)_{n}+(b^2)_{n}.
\end{equation*}
According to the the structure of the triangle \hpt4q and the expansion of $(x+y)^2$, we will also use the sums
$(ab)_{n}$, $(ba)_{n}$ and $(bb)_{n}$.

First consider the sum $(a^2)_{n+1}$. Observe that an element of type $A$ in row $n+1$ is a sum of two elements either of a type $A$ and $B$ or of a type $B$ from row $n$. Since, apart from the wingers, each element of row $n$ takes part twice in constructing the elements of type $A$ in row $n+1$, then 
\begin{eqnarray*}
(a^2)_{n+1}&=& 2(a^2)_{n}+{2\choose1}(ab)_{n} +{2\choose1}(ba)_{n} +{2\choose1}(bb)_{n}+2(b^2)_{n}-2 \\
&=& 2(a^2)_{n}+2(ab)_{n} +2(ba)_{n} +2(bb)_{n} +2(b^2)_{n}-2.
\end{eqnarray*}

Clearly, an element of type $B$ in row $n+1$ coincides an element of type either $A$ ($(q-4)$-times) or $B$ ($(q-3)$-times, apart from the wingers) from row $n$, thus we see
$$
(b^2)_{n+1}= (q-4)(a^2)_{n}+(q-3)(b^2)_{n}-2(q-4),
$$
where $-2(q-4)$ is again the correction caused by the wingers.

To explain
$$
(ab)_{n+1}= (a^2)_{n}+(ab)_{n} +(ba)_{n} +(bb)_{n}+(b^2)_{n} -1,
$$
we put together the construction rule of elements type $A$ and $B$, respectively, and the fact we need to consider the neighbour pairs $A,B$ (in type). Since the left winger 1 does not appear in $(ab)_{n+1}$, we have the correction $-1$. By the vertical symmetry, 
the relation 
\begin{equation}\label{abba}
(ab)_{n+1}=(ba)_{n+1}
\end{equation}
holds. 

Finally,
$$
(bb)_{n+1}=  (q-5)(a^2)_{n}+(q-4)(b^2)_{n}-2 (q-4)
$$
holds since two neighbours of type $B$ are  coming from an element of type either $A$ or $B$ in direct manner.

As a summary, using (\ref{abba}) we have the system    

\begin{equation}\label{eq:SUM_seq01}
  \begin{array}{ccl}
(a^2)_{n+1}&=& 2(a^2)_{n}+4(ab)_{n} +2(b^2)_{n}+2(bb)_{n}-2,\\
(ab)_{n+1}&=& (a^2)_{n}+2(ab)_{n} +(b^2)_{n}+(bb)_{n} -1,\\
(b^2)_{n+1}&=& (q-4)(a^2)_{n}+(q-3)(b^2)_{n}-2 (q-4),\\
(bb)_{n+1}&=&  (q-5)(a^2)_{n}+(q-4)(b^2)_{n}-2 (q-4).
  \end{array} 
\end{equation}

\subsection{General case}

Recall that $(b^{k-1}b)_n=(b^{k-2}b^2)_n=\cdots =(bb^{k-1})_n$. In the sequel, we will denote them by $u_n$.
We use analoguously the considerations and the explanations from the previous subsection (where the case $k=2$ was handled), and together with the binomial theorem we gain the system of recurrence equations 

\begin{eqnarray}\label{eq:SUMk_seq01}
(a^{k})_{n+1}&=& \sum_{i=0}^{k}{{k}\choose{i}}(a^{k-i}b^{i})_{n} +\sum_{i=0}^{k}{{k}\choose{i}}(b^{k-i}a^{i})_{n}+ \left( \sum_{i=0}^{k}{{k}\choose{i}}-2 \right) u_n -2, \nonumber\\    
 (a^{k-j}b^{j})_{n+1}&=& \sum_{i=0}^{k-j}{{k-j}\choose{i}}(a^{k-j-i}b^{j+i})_{n} + \sum_{i=0}^{k-j}{{k-j}\choose{i}}(b^{k-j-i}a^{j+i})_{n}  \nonumber \\
 & & \qquad + \left( \sum_{i=0}^{k-j}{{k-j}\choose{i}}-1 \right) u_n -1,\\
(b^k)_{n+1} &=& (q-4)(a^k)_{n}+(q-3)(b^k)_{n}-2 (q-4), \nonumber\\  
 u_{n+1} &=& (q-5)(a^k)_{n}+(q-4)(b^k)_{n}-2 (q-4) \nonumber
\end{eqnarray}
with $k+2$ sequences, where $j=1,  \dots, k-1$. To be more cut-clear, we equivalently have
\begin{eqnarray}\label{eq:SUMk_seq02}
(a^{k})_{n+1}&=& 2(a^{k})_{n} + 2\sum_{i=1}^{k-1}{{k}\choose{i}}(a^{k-i}b^{i})_{n}+ 2(b^{k})_{n} +(2^k-2)u_n-2,\nonumber \\    
(a^{k-j}b^{j})_{n+1}&=& (a^{k})_{n}+ \sum_{i=0}^{k-j-1}{{k-j}\choose{i}}(a^{k-j-i}b^{j+i})_{n} 
+ \sum_{i=0}^{k-j-1}{{k-j}\choose{i}}(a^{j+i}b^{k-j-i})_{n} \nonumber \\
&&\qquad\qquad + (b^{k})_{n}+ (2^{k-j}-1) u_n -1,\\
(b^k)_{n+1} &=& (q-4)(a^k)_{n}+(q-3)(b^k)_{n}-2 (q-4), \nonumber \\ 
u_{n+1} &=& (q-5)(a^k)_{n}+(q-4)(b^k)_{n}-2 (q-4). \nonumber
\end{eqnarray} 

Now we claim to eliminate the sequences $(a^{k})_{n}$ and $(b^{k})_{n}$ from (\ref{eq:SUMk_seq02}).

\subsection{Background to eliminate $(a^{k})_{n}$ and $(b^{k})_{n}$ from the system (\ref{eq:SUMk_seq02})}

Our purpose is to eliminate the sequences $(a^{k})_{n}$ and $(b^{k})_{n}$ by giving a recursive formula to describe them.
In the next part, we develope a tool for handling such problems in general. System (\ref{eq:SUMk_seq02}) can be interpreted as a vector recursion of the form 
$$\g_{t+1}=\mathbf{M}\g_t+\bar{\bf{h}}$$ 
given by a suitable matrix $\mathbf{M}$ and vector $\bar{\bf{h}}$, when we consider $(a^{k})_{n},(a^{k-1}b)_{n},\dots,(b^{k})_{n}, u_n$ as coordinate sequences of $\g_{n}$. First we concentrate on the homogenous case, i.e.~when $\bar{\bf{h}}$ is the zero vector  $\bar{\bf{0}}$.

Let $\nu\ge2$ be an integer, further let $\g_0,\dots,\g_{\nu-1}\in\R^\nu$ denote linearly independent initial vectors.
With the coefficients $\alpha_0,\dots,\alpha_{\nu-1}\in\R$ we can set up the homogenous linear vector recursion
\begin{equation} \label{startrec}
\g_{t}=\alpha_{\nu-1}\g_{t-1}+\dots+\alpha_0\g_{t-\nu}\qquad (t\ge\nu).
\end{equation}

\begin{lemma} \label{thM}
There exists uniquely a matrix $\mathbf{M}\in\R^{\nu\times \nu}$ such that
\begin{equation} \label{M}
\g_{t+1}=\mathbf{M}\g_t
\end{equation}
holds for any $t\in\N$.

Moreover, if (\ref{M}) holds for a vector sequence $\g_t$ with a given (not necessarily regular) matrix $\mathbf{M}\in\R^{\nu\times \nu}$, then $\g_t$ satisfies (\ref{startrec}), where the coefficients coincide the negative of the coefficients of the characteristic polynomial of $\mathbf{M}$.

\end{lemma}

\begin{proof}

Obviously, by (\ref{startrec}) one can obtain $\g_\nu\in\R^\nu$.
Put $\mathbf{G}=[\g_0,\dots,\g_{\nu-1}]\in\R^{\nu\times \nu}$ and  $\mathbf{G}^\star=[\g_1,\dots,\g_{\nu}]\in\R^{\nu\times \nu}$, where the matrix $\mathbf{G}$ is clearly regular.

First we show that the only possibility is $\mathbf{M}=\mathbf{G}^\star \mathbf{G}^{-1}$.
Assume that $\mathbf{M}$ satisfies (\ref{M}) if
$t=0,\dots,\nu-1$. The system of vector equations
\begin{eqnarray*}
\g_1 & = & \mathbf{M}\g_0 \\
     & \vdots &   \\
\g_{\nu} & = & \mathbf{M}\g_{\nu-1} \\
\end{eqnarray*}
is equivalent to the matrix equation
\begin{equation*} \label{Meq}
\mathbf{G}^{\star}=\mathbf{MG}.
\end{equation*}
But $\mathbf{G}$ is regular, therefore $\mathbf{M}=\mathbf{G}^\star \mathbf{G}^{-1}$ really exists.

Now we justify that $\mathbf{M}=\mathbf{G}^\star \mathbf{G}^{-1}$ is suitable for arbitrary subscript $t$, i.e.~(\ref{M}) holds for  any $t\in\N$. The definition of the matrix $\mathbf{M}$ proves the statement for
$t=0,1,\dots,\nu-1$. In order to show for arbitrary $t$, we use the technique of induction. Hence suppose that
(\ref{M}) holds for $t=0,1,\dots,\tau$. Then
\begin{eqnarray*}
\g_{\tau+1}&=&\alpha_{\nu-1}\g_{\tau}+\dots+\alpha_0\g_{\tau-\nu+1}=
\alpha_{\nu-1}\mathbf{G}^\star \mathbf{G}^{-1}\g_{\tau-1}+\dots+\alpha_0\mathbf{G}^\star \mathbf{G}^{-1}\g_{\tau-\nu} \\
&=& \mathbf{G}^\star \mathbf{G}^{-1}(\alpha_{\nu-1}\g_{\tau-1}+\dots+\alpha_0\g_{\tau-\nu})=\mathbf{G}^\star \mathbf{G}^{-1}\g_{\tau}.
\end{eqnarray*}

Finally, we prove that if $\mathbf{M}$ satisfies (\ref{M}), then we arrive at (\ref{startrec}) with the given conditions.
Let $k(x)=x^\nu-\alpha_{\nu-1}x^{\nu-1}-\dots-\alpha_0$ denote the
characteristic polynomial of $\mathbf{M}$. By the Cayley--Hamilton theorem, $k(\mathbf{M})$ is zero a matrix, consequently
$k(\mathbf{M})\g_0=\0$, i.e.
\begin{equation} \label{m1}
\mathbf{M}^\nu\g_0-\alpha_{\nu-1}\mathbf{M}^{\nu-1}\g_0-\dots-\alpha_0\g_0=\0.
\end{equation}
Since for any natural number $t$ the equality $\g_{t+1}=\mathbf{M}\g_t$ implies $\g_{t+1}=\mathbf{M}^{t+1}\g_0$, (\ref{m1})
can be rewritten as
\begin{equation*} \label{m2}
\g_\nu=\alpha_{\nu-1}\g_{\nu-1}+\dots+\alpha_0\g_0.
\end{equation*}
Further 
\begin{equation*} \label{m3}
\g_{\nu+1}=\mathbf{M}\g_\nu=\mathbf{M}(\alpha_{\nu-1}\g_{\nu-1}+\dots+\alpha_0\g_0)=\alpha_{\nu-1}\g_{\nu}+\dots+\alpha_0\g_1,
\end{equation*}
and similarly
$$\g_{t+1}=\mathbf{M}\g_t=\alpha_{\nu-1}\g_{t}+\dots+\alpha_0\g_{t-\nu}\qquad (t>\nu) $$
holds. 
\end{proof}
\begin{remark}
In this lemma, we proved a bit more than we need to investigate the power sums in hyyperbolic Pascal triangles. Indeed, we will use only the direction which provides the coefficients $\alpha_i$ from the characteristic polynomial of $\mathbf{M}$.
A statement of \cite{NL_hyppyr} also proves the second part of Lemma \ref{thM}.
\end{remark}

\subsection{Connection between the homogenous and inhomogenous vector recurrences}

Since system (\ref{eq:SUMk_seq02}) is inhomogenous, and Lemma \ref{thM} is able to handle only homogenous vector recursions, therefore we must clarify the transit between them. 

Assume that the vector $\bar{\bf{h}}$ is fixed, and we have
$$
\g_{t+1}=\mathbf{M}\g_t+\bar{\bf{h}} \qquad (t\in\N),
$$
further suppose that the characteristic polynomial of $\mathbf{M}$ is $k(x)=x^\nu-\alpha_{\nu-1}x^{\nu-1}-\dots-\alpha_0$. 
Obviously, 
$$
\g_{t+2}=\mathbf{M}\g_{t+1}+\bar{\bf{h}},
$$
and 
$$
\g_{t+2}-\g_{t+1}=\mathbf{M}(\g_{t+1}-\g_{t}).
$$
Put $\d_{t}=\g_{t+1}-\g_{t}$. Thus $\d_{t+1}=\mathbf{M}\d_{t}$ implies
$$
\d_{t}=\alpha_{\nu-1}\d_{t-1}+\cdots+\alpha_0\d_{t-\nu},
$$
and then
$$
\g_{t+1}=(\alpha_{\nu-1}+1)\g_{t}+(\alpha_{\nu-2}-\alpha_{\nu-1})\g_{t-1}+\cdots+(\alpha_0-\alpha_1)\g_{t+1-\nu}-\alpha_0\g_{t-\nu}.
$$
Finally, observe that
$$
(x-1)k(x)=x^{\nu+1}-(\alpha_{\nu-1}+1)x^\nu-(\alpha_{\nu-2}-\alpha_{\nu-1})x^{\nu-1}-\cdots-(\alpha_0-\alpha_1)x+\alpha_0,
$$
where $k(x)$ is the characteristic polynomial of $\mathbf{M}$.

\subsection{Determination of the characteristic polynomial}

Consider now the matrix $\mathbf{M}=[m_{i,j}]\in\N^{(k+2)\times(k+2)}$ generated by the system (\ref{eq:SUMk_seq02}). If we introduce binomial coefficients with negative lower index, or when the lower index is greater then the upper index, and take such a value zero (this is one of the conventional approaches), the left upper $k\times (k+1)$ minor matrix can simply be given. Indeed, if $ 0\le i\le k-1$, $0\le j\le k$, by (\ref{eq:SUMk_seq02}), 
$$
m_{i,j}={k-i\choose k-i-j}+{k-i\choose i-j}={k-i\choose j}+{k-i\choose k-j}. 
$$
Further $m_{0,k+1}=2^k-2$, and if  $1\le i\le k-1$ and $j=k+1$, then  $m_{i,j}=2^{k-i}-1$. Moreover, $m_{k,0}=q-4$, $m_{k,k}=q-3$, $m_{k+1,0}=q-5$, $m_{k+1,k}=q-4$. The entries have not been listed up here ($m_{i,j}$ with $k\le i\le k+1$, $1\le j\le k+1,\; j\ne k$) are zero.
Obviously (see (\ref{eq:SUMk_seq02})), in our case the vector
$$
\bar{\bf{h}}=\left[
\begin{array}{ccccccc}
     -2 & -1 & -1 & \cdots & -1 & -2(q-4) & -2(q-4) \\	
\end{array}\right]^\top.
$$

By visualizing the informations we gain
$$
\mathbf{M}=\left[
\begin{array}{ccccc:c}
 \ddots &    &  &  & \Ddots & 2^k-2 \\	
        &    &  &  &        & 2^{k-1}-1 \\	
		  &   & {k-i\choose j}+{k-i\choose k-j}  &  &        & \vdots \\	      
			&    &  &  &        & 2^2-1\\
	\Ddots		&   &  &  &    \ddots    & 2^{1}-1 \\ \hdashline
   q-4     & 0  & \cdots & 0 & q-3  & 0 \\	
		q-5     & 0  & \cdots & 0 & q-4  & 0 \\	
\end{array}\right].
$$

As usual, we will use the formula $k(x)=\det(x\mathbf{I}-\mathbf{M})$ for determining the characteristic polynomial of the matrix $\mathbf{M}$, where $\mathbf{I}$ is the unit matrix. We carry out the evaluation in two steps, finally we obtain a quite simple-looking determinant.

First we deal with $\mathbf{M}$ itself. Let $R_u$ denote the row $u$ of the matrix $\mathbf{M}$.
The following lemma describes row equivalent transformations of $\mathbf{M}$, which keeps the value of the determinant. Before applying the lemma, first we make a preparation step (again a row equivalent transformation) for modifying $\mathbf{M}$, namely $R_{k}^{new}=R_k-R_{k+1}$ to result the penultimate row
$$
\left[
\begin{array}{ccccc:c}
     1 & 0  & \cdots & 0 & 1 & 0 \\	
\end{array}\right].
$$

\begin{lemma}
	For $u=0, 1, \ldots, k-1$ apply successively the row equivalent operation
	\begin{equation}\label{S}
	R_u^{new}=\sum_{t=0}^\delta (-1)^t{\delta\choose t}R_{u+t},
	\end{equation}
where $\delta=k-u\ge1$. Then the elements of the left upper $(k+1)\times(k+1)$ minor of $\mathbf{M}^{new}$ are zero, except the main diagonal, and the antidiagonal elements which are 1. If the main diagonal and antidiagonal meet, the common element is 2.

Further in the last column we have $m_{i,k+1}^{new}=1$ if $1\le i\le k-1$, and $m_{0,k+1}^{new}=0$.  
\end{lemma}

\begin{proof}
Let $j$ be arbitrarily fixed. Then
\begin{eqnarray*}
m_{u,j}^{new}&=&\sum_{t=0}^\delta (-1)^t{\delta\choose t}\left({k-u-t\choose j}+{k-u-t\choose k-j}\right)\\
&=& \sum_{t=0}^\delta (-1)^t{\delta\choose t}{k-u-t\choose j}+\sum_{t=0}^\delta (-1)^t{\delta\choose t}{k-u-t\choose k-j} \\
&=& {k-u-\delta\choose j-\delta}+{k-u-\delta\choose k-j-\delta},
\end{eqnarray*}
where the last equality is a consequence of the more general identity 
$$\sum_{t=0}^\delta (-1)^t{\delta\choose t}{z-t\choose r}={z-\delta\choose r-\delta}$$ 
(see, for instance, page 28 in \cite{G}).
Clearly, $\delta=k-u$ implies
$$
{k-u-\delta\choose j-\delta}+{k-u-\delta\choose k-j-\delta}=
\left\{
\begin{array}{ll}
     0, & \rm{if}\; {\it j\ne k-u}\; {\rm and}\; {\it j\ne u};\\
     1, & \rm{if}\; either\; {\it j=k-u}\; or\; {\it j= u}; \\
		 2, & \rm{if}\; {\it j=k-u}\; and\; {\it j= u}.\\
\end{array}\right.
$$
Note that the last case can be occurred only if $k=2u$.

Continuing with the last column, for $1\le u<k$ we easily conclude
\begin{eqnarray*}
m_{u,k+1}^{new}&=&\sum_{t=0}^\delta (-1)^t{\delta\choose t}\left(2^{k-u-t}-1\right)
   = \sum_{t=0}^\delta (-1)^t{\delta\choose t}2^{k-u-t}-\sum_{t=0}^\delta (-1)^t{\delta\choose t} \\
&=& 2^{k-u}\sum_{t=0}^\delta (-1)^t{\delta\choose t}\left(\frac{1}{2}\right)^t-0
   = 2^{k-u}\cdot \left(1-\frac{1}{2}\right)^\delta=1.
\end{eqnarray*}
Clearly, the same argument leads to 0 in case of the last element of the $0$th row.
\end{proof}

Now let $\mathbf{\tilde M}=\mathbf{M}-x\mathbf{I}$, and we will carry out the same row equivalent operations on $\mathbf{\tilde M}$ what we did on $\mathbf{M}$. Only one modification is, that after the preparation step we insert an additional operation. Namely, $R_{k}^{new}=R_k-R_{k+1}$ is followed by
$R_{k+1}^{new}=R_{k+1}-(q-5)R_{k}$. Hence keeping tabs on the last two rows of $\mathbf{\tilde M}$, we see
$$
\left[
\begin{array}{ccccc:c}
     q-4 & 0 & \cdots & 0 & (q-3)-x & 0 \\
		 q-5 & 0 & \cdots & 0 & q-4 & -x\\
\end{array}
\right]\Longrightarrow
\left[
\begin{array}{ccccc:c}
     1 & 0 & \cdots & 0 & 1-x & x \\
		 0 & 0 & \cdots & 0 & 1+(q-5)x & -(q-4)x\\
\end{array}
\right].
$$

Since $\mathbf{\tilde M}=\mathbf{M}-x\mathbf{I}$ effects the elements only in the diagonal of $\mathbf{M}$ (by $-x$), according to 
(\ref{S}),  for $0\le u<j\le k$ we find
$$
\tilde m_{u,j}^{new}=(-1)^{j-u+1}{\delta\choose j-u}x.
$$
Finally, the last column of $\mathbf{\tilde M}^{new}$ is given by the vector
$$
\left[
\begin{array}{ccccccc}
     (-1)^kx \;&\; 1+(-1)^{k-1}x \;&\; \cdots \;&\; 1+x \;&\; 1-x \;&\; x \;&\; -(q-4)x \\	
\end{array}\right]^\top.
$$
Hence, we conclude that $\mathbf{\tilde M}^{new}$ is the sum of
{\renewcommand{\arraystretch}{1.5}
$$
\left[
\begin{array}{rrrrrr:r}
     -x & {k\choose 1}x & -{k\choose 2}x & \quad\cdots\quad & (-1)^k{k\choose k-1}x & (-1)^{k+1}{k\choose k}x & (-1)^kx \\
		0 & -x & {k-1\choose 1}x & \quad\cdots\quad & (-1)^{k-1}{k-1\choose k-2}x & (-1)^{k}{k-1\choose k-1}x & (-1)^{k-1}x\\
		0 & 0 & -x & \quad\cdots\quad & (-1)^{k-2}{k-2\choose k-3}x & (-1)^{k-1}{k-2\choose k-2}x & (-1)^{k-2}x \\
		\vdots & \vdots & \vdots & \qquad\ddots\quad & \vdots & \vdots & \vdots \\
		0 & 0 & 0 & \quad\cdots\quad & -x & {1\choose 1}x & -x \\ \hdashline
		0 & 0 & 0 & \quad\cdots\quad & 0 & -x & x \\ 
		0 & 0 & 0 & \quad\cdots\quad & 0 & (q-5)x & -(q-4)x 
\end{array}\right]
$$
and
$$
\left[
\begin{array}{cccccc:c}
    1 & 0 & 0 & \quad\cdots\quad & 0 & 1 & 0 \\
		0 & 1 & 0 & \quad\cdots\quad & 1 & 0 & 1 \\
		0 & 0 & 1 & \quad\cdots\quad & 0 & 0 & 1 \\
		\vdots & \vdots & \vdots & \quad\ddots\quad & \vdots & \vdots & \vdots \\
		0 & 1 & 0 & \quad\cdots\quad & 1 & 0 & 1 \\ \hdashline
		1 & 0 & 0 & \quad\cdots\quad & 0 & 1 & 0 \\
		0 & 0 & 0 & \quad\cdots\quad & 0 & 1 & 0 
\end{array}\right].
$$
}
At the end, the characteristic polynomial $k(x)=(-1)^k\det(\mathbf{\tilde M}^{new})$, and by the previous subsection we can return the inhomogenous case  with the polynomial
$$
(x-1)k(x)=(x-1)(-1)^k\det(\mathbf{\tilde M}^{new}).
$$
In the range $1\le k\le 11$ the appendix will provide the results, since knowing the characteristic polynomial, one can obtain directly the corresponding recursive formula.

\subsection{Example: return to the case $k=2$}

System (\ref{eq:SUM_seq01}) provides the matrix $\mathbf{M}$, and then 
$$
\mathbf{\tilde M}=\mathbf{M}-x\mathbf{I}=\left[
\begin{array}{ccc:c}
    2-x & 4 & 2 & 2 \\
		1 & 2-x & 1 & 1\\ \hdashline
		q-4 & 0 & q-3-x & 0 \\ 
		q-5 & 0 & q-4 & -x 
\end{array}\right].
$$
The consecutive row equivalent transformations $R_2=R_2-R_3$, $R_3-(q-5)R_2$, and then $R_0=R_0-{2\choose1}R_1+{2\choose2}R_2$, $R_1=R_1-{1\choose1}R_2$ return with
 
$$
\mathbf{\tilde M}^{new}=\left[
\begin{array}{ccc:c}
    -x & 2x & -x & x\\
		0 & -x & x & -x\\ \hdashline
		0 & 0 & -x & x \\ 
		0 & 0 & (q-5)x & -(q-4)x 
\end{array}\right]+\left[
\begin{array}{ccc:c}
    1 & 0 & 1 & 0\\
		0 & 2 & 0 & 1\\ \hdashline
		1 & 0 & 1 & 0 \\ 
		0 & 0 & 1 & 0 
\end{array}\right]
$$
Thus the determinant
$$
k_1(x)=\det(\mathbf{\tilde M}^{new})=x^4-(q+1)x^3+6x^2-2x,
$$
and then we obtain 
$$
k(x)=(x-1)(-1)^2k_1(x)=x^5-(q+2)x^4+(q+7)x^3-8x^2+2x,
$$
which provide the recursive rule 
$$
(s^2)_{n}=(q+2)(s^2)_{n-1}-(q+7)(s^2)_{n-2}+8(s^2)_{n-3}-2(s^2)_{n-4}.
$$
The initial values can be given from the first 4 rows of ${\cal HPT}_{\!4,q}$:
$(s^2)_{1}=2$, $(s^2)_{2}=6$, $(s^2)_{3}=4q+4$, $(s^3)_{4}=4q^2+6q-20$.

\section{Appendix and conjectures}

The method works for arbitrary positive integer $k$. In the next table we collect the coefficients of
$$
(s^k)_{n}=\sum_{j=1}c_j(q)(s^k)_{n-j}.
$$

{\footnotesize  
{\renewcommand{\arraystretch}{1.4}
\begin{center} \label{table1}
\begin{tabular}{||c||c|c|c|c|c||}
  \hline\hline

  $k$ & $c_1(q)$ & $c_2(q)$ & $c_3(q)$ & $c_4(q)$ & $c_5(q)$   \\  \hline\hline
	
	0 & $q-1$ & $-q+1$ & $1$ &  &    \\  \hline
	
	1 & $q$ & $-q-1$ & $2$ &  &     \\  \hline
	
	2 & $q+2$ & $-q-7$ & $8$ & $-2$ &    \\  \hline
	
	3 & $q+4$ & $q-19$ & $-2q+18$ & $-2$ &   \\  \hline
	
	4 & $q+7$ & $6q-41$ & $-7q+31$ & $6$ & $-2$   \\  \hline
	
	5 & $q+11$ & $18q-71$ & $-9q-17$ & $-10q+88$ & $-10$ \\  \hline
	
	6 & $q+17$ & $44q-99$ & $17q-303$ & $-62q+404$ & $-14$  \\  \hline
	
	7 & $q+26$ & $99q-68$ & $177q-1400$ & $-235q+1183$ & $-42q+302$  \\  \hline
	
	8 & $q+40$ & $213q+206$ & $757q-4842$ & $-717q+2981$ & $-254q+1782$   \\  \hline
	
	9 & $q+62$ & $447q+1288$ & $2433q-15116$ & $-2431q+10555$ & $-450q+3662$   \\  \hline
	
	10 & $q+97$ & $924q+4782$ & $6096q-48560$ & $-13946q+75623$ & $5903q-24351$ \\  \hline
	
	11 & $q+153$ & $1892q+15110$ & $9924q-181164$   & $-99514q+592503$ & $78523q-360899$  \\  \hline\hline	
  
\end{tabular}
\end{center}
}
}

{\footnotesize  
{\renewcommand{\arraystretch}{1.4}
\begin{center} \label{table2}
\begin{tabular}{||c||c|c|c||}
  \hline\hline
  $k$ &  $c_6(q)$&  $c_7(q)$ & $c_8(q)$  \\  \hline\hline
	6 & $-4$ &  & \\  \hline
	7 & $-42$  &  & \\  \hline
	8 & $-162$ & $-4$ & \\  \hline
	9 & $-450$  & 0 & \\  \hline
	10 & $1022q-8408$ & $814$ & $4$\\  \hline
	11 & $9174q-74876$ & $9174$& 0 \\  \hline\hline	
\end{tabular}
\end{center}
}
}
\medskip

Partially from the results we have the following
\begin{conj}
Let $k$ be given. Then 
\begin{itemize}
	\item the linear recurrence corresponding to $k$ has order $\left\lfloor k/2\right\rfloor+3$.
	\item Further the coefficients $c_j(q)$ are linear polynomials in $q$.
\end{itemize}
\end{conj}

Our conjecture is based not only on the first few cases but also on the following approach. We are able to decrease the size of \eqref{eq:SUMk_seq01} which entails the reduction of the dimension of matrix $\mathbf{M}$ (but we loose its simplicity). The new system of $\left[ k/2\right]+3$ recurrent sequences is the following.

Put $\ell=\lfloor(k-1)/2\rfloor$ and $m=\lceil(k-1)/2\rceil$. Obviously, if $k$ is odd, then $m=\ell$, otherwise $m=\ell+1$. Let $(c_j)_n=(a^{k-j}b^{j})_{n}+(a^{j}b^{k-j})_{n}$, where $j=1,\ldots, \ell$. Moreover, if $k$ is even, then let $(c_{\ell+1})_n=(a^{\ell+1}b^{\ell+1})_{n}$. Now the new system of recurrences admits
\begin{eqnarray*}\label{eq:SUMk_seq05}
(a^{k})_{n+1}&=& 2(a^{k})_{n} +      2\sum_{i=1}^{m}{{k}\choose{i}}(c_i)_{n} + 2(b^{k})_{n}+(2^k-2)u_n-2,\nonumber\\
(b^k)_{n+1} &=& (q-4)(a^k)_{n}+(q-3)(b^k)_{n}-2 (q-4),\nonumber\\      
(c_{j})_{n+1}&=& (a^{k})_{n}+ \sum_{i=j}^{m}{{k-j}\choose{i-j}}(c_{i})_{n}  
  +   \sum_{i=1}^{m} \left(  {{k-j}\choose{i}}+ {{j}\choose{i}} \right)
(c_{i})_{n} \\ 
& & \qquad+ (b^{k})_{n}
+ (2^{k-j}-1) u_n -1,\nonumber\\
u_{n+1} &=& (q-5)(a^k)_{n}+(q-4)(b^k)_{n}-2 (q-4).\nonumber
\end{eqnarray*} 
\noindent where $j=1,\dots, \ell$ . When $k$ is even, thenthere is an addition sequence
\begin{eqnarray*}\label{eq:SUMk_add}
(c_{\ell+1})_{n+1}&=& (a^{k})_{n}+ (c_{\ell+1})_{n}  
  + \sum_{i=1}^{\ell+1}{{k-\ell-1}\choose{i}}(c_{i})_{n}\nonumber \\
& &\qquad+ (b^{k})_{n}+ (2^{k-\ell-1}-1) u_n -1.
\end{eqnarray*}

\end{document}